\documentclass[a4paper]{amsart}
\usepackage{amssymb, amsmath, amscd, amsthm}
\usepackage{amsfonts}
\usepackage{psfrag}
\usepackage[all]{xy}
\usepackage{graphicx}

\newtheorem{thm}{Theorem}[section]

\newtheorem{lem}[thm]{Lemma}
\newtheorem{prop}[thm]{Proposition}
\theoremstyle{definition}

\newtheorem{defn}[thm]{Definition}
\theoremstyle{remark}

\newtheorem{ex}[thm]{Example}
\newcommand{\s}{\psi}
\newcommand{\p}{\varphi}
\newcommand{\R}{\mathbb{R}}
\newcommand{\N}{\mathbb{N}}

\newcommand{\fr}{\vec \varphi}
\newcommand{\rank}{{\mathrm{rank}}}

\begin{document}

\title{Multidimensional persistent homology is stable}

\author[A. Cerri]{A. Cerri}
\author[B. Di Fabio]{B. Di Fabio}
\author[M. Ferri]{M. Ferri}
\author[P. Frosini]{P. Frosini}
\address{ARCES and Dipartimento di Matematica, Universit\`a di Bologna, Italia}
\email{\{cerri,difabio,ferri,frosini\}@dm.unibo.it}
\thanks{Research carried out under the auspices of
INdAM-GNSAGA. The last author partially carried out the research
within the activity of ARCES ``E. De Castro'', University of
Bologna.}
\author[C. Landi]{C. Landi}
\address{Dipartimento di Scienze e Metodi dell'Ingegneria,
Universit\`a di Modena e Reggio Emilia, Italia}
\email{clandi@unimore.it}



\begin{abstract}
Multidimensional persistence studies topological features of
shapes by analyzing the lower level sets of vector-valued
functions. The rank invariant completely determines the
multidimensional analogue of persistent homology groups. We prove
that multidimensional rank invariants are stable with respect to
function perturbations. More precisely, we construct a distance
between rank invariants such that small changes of the function
imply only small changes of the rank invariant. This result can be
obtained by assuming the function to be just continuous.
Multidimensional stability opens the way to a stable shape
comparison methodology based on multidimensional persistence.
\end{abstract}

\maketitle

\section*{Introduction}
The study of the topology of  data is attracting more and more
attention from the mathematical community. This challenging
subject of research is motivated by the large amount of scientific
contexts where it is required to deal with qualitative geometric
information. Indeed, the topological approach allows us to greatly
reduce the complexity of the data by focusing the analysis just on
their relevant part. This research area is widely discussed in
\cite{BiDe*08,Ca09}.

In this line of research, persistent homology has been revealed to
be a key mathematical method for studying the topology of data,
with applications in an increasing number of fields, ranging from
shape description (e.g.,
\cite{CaZo*05,CeFeGi06,MoSaSa08,VeUrFrFe93}) to data
simplification \cite{EdLeZo02} and hole detection in sensor
networks \cite{DeGh07}. Recent surveys on the topic include
\cite{EdHa08,Gh08,Zo05}. Persistent homology describes topological
events occurring through the filtration of a topological space
$X$. Filtrations are usually expressed by real functions
$\p:X\to\R$. The main idea underlying this approach is that the
most important piece of information enclosed in geometrical data
is usually the one that is ``\emph{persistent}'' with respect to
small changes of the function defining the filtration.

The analysis of persistent topological events in the lower level
sets of the functions (e.g., creation, merging, cancellation of
connected components, tunnels, voids) is important for capturing a
global description of the data under study. These events can be
encoded in the form of a parameterized version of the Betti
numbers, called a {\em rank invariant} \cite{CaZo09} (already
known in the literature as a {\em size function} for the $0$th
homology \cite{FrLa99,KaMiMr04,VeUrFrFe93}).

Until recently, research on persistence has mainly focused on the
use of scalar functions for describing filtrations. The extent to
which this theory can be generalized to a situation in which two
or more functions characterize the data is currently under
investigation \cite{BiCe*08,CaZo09}. This generalization to
vector-valued functions is usually known as the
\emph{Multidimensional Persistence Theory}, where the adjective
multidimensional refers to the fact that filtrating functions are
vector-valued, and has no connections with the dimensionality of
the space under study. The use of vector-valued filtrating
functions, introduced in \cite{FrMu99}, would enable the analysis
of richer data structures.

One of the most important open questions in current research about
multidimensional persistent homology concerns the \emph{stability
problem}. In plain words, we need to determine how the computation
of invariants in this theory is affected by the unavoidable
presence of noise and approximation errors. Indeed, it is clear
that any data acquisition is subject to perturbation and, if
persistent homology were not stable, then distinct computational
investigations of the same object could produce completely
different results. Obviously, this would make it impossible to use
such a mathematical setting in real applications.

In this paper we succeed in solving this problem and giving a
positive answer about the stability of multidimensional persistent
homology (Theorem \ref{Multidimensional}). More precisely, we
prove that the rank invariants of nearby vector-valued filtrating
functions are ``close to each other'' in the sense expressed by a
suitable matching distance. We point out that, until now, the
stability of persistent homology has been studied only for
scalar-valued filtrating functions.

Our stability result for the multidimensional setting requires us
to use some recently developed ideas to investigate
\emph{Multidimensional Size Theory} \cite{BiCe*08}. The proof of
multidimensional stability is obtained by reduction to the
one-dimensional case, via an appropriate foliation in half-planes
of the domain of the multidimensional rank invariants, and the
definition of a family of suitable and possibly non-tame
continuous scalar-valued filtrating functions. Our result follows
by proving the stability property in this one-dimensional case. We
observe that the results obtained in \cite{CoEdHa07}, for tame
scalar functions, and in \cite{CaCo*09}, under some finiteness
assumptions, cannot be applied here. Indeed, we underline that our
stability result requires the functions only to be continuous. Our
generalization of one-dimensional stability from tame to
continuous functions is a positive answer to the question risen in
\cite{CoEdHa07}.

Our main point is that all the proofs carried out in \cite{FrLa01}
to analyze size functions via diagrams of points, the so called
persistence diagrams, as well as those in \cite{dAPhD02,dAFrLa} to
prove the stability of size functions, can be developed in a
completely parallel way for rank invariants associated with
continuous functions. As a consequence, we do not repeat the
technical details of all the proofs, whenever the arguments are
completely analogous to those used in the literature about size
functions. For the same reason, we refer the reader to
\cite{BiCe*08} (see also \cite{CaDiFe07}) for the proofs of the
Reduction Theorem \ref{Reduction} and Theorem \ref{Stability}
(Stability w.r.t. Function Perturbations) used to deduce our main
result about the stability of multidimensional persistent
homology.

Finally, we emphasize that throughout this paper we  work with
\v{C}ech homology over  a field. In the framework of persistence,
\v{C}ech homology has already been considered by Robbins in
\cite{Ro99,Ro00}. In our opinion, a strong motivation for using
\v{Cech homology} in studying persistent homology groups is that,
having the continuity axiom,  it ensures that the rank invariant
can be completely described by a persistence diagram, unlike the
singular and simplicial theories that guarantee that such a
description is complete only outside a set of vanishing measure,
as explained in Section \ref{Sect:Right-continuity}. We point out
that the \v{C}ech approach to homology theory is currently being
investigated for computational purposes \cite{Mr}.

\section{Basic Definitions and the Main Result}

In this paper, each considered space is assumed to be
triangulable, i.e. there is a finite simplicial complex with
homeomorphic underlying space. In particular, triangulable spaces
are always compact and metrizable.

The following relations $\preceq$ and $\prec$ are defined in
$\R^n$: for $\vec u=(u_1,\dots,u_n)$ and $\vec v=(v_1,\dots,v_n)$,
we say $\vec u\preceq\vec v$ (resp. $\vec u\prec\vec v$) if and
only if $u_i\leq\ v_i$ (resp. $u_i<v_i$) for every index
$i=1,\dots,n$. Moreover, $\R^n$ is endowed with the usual
$\max$-norm: $\|(u_1,u_2,\dots,u_n)\|_{\infty}=\max_{1\leq i\leq
n}|u_i|$.

We shall use the following notations: $\Delta^+$ will be the open
set $\{(\vec u,\vec v)\in\R^n\times\R^n:\vec u\prec\vec v\}$.
Given a triangulable space $X$, for every $n$-tuple $\vec
u=(u_1,\dots,u_n)\in\R^n$ and for every function
$\vec\p:X\to\R^n$, we shall denote by $X\langle\fr\preceq \vec
u\,\rangle$ the set $\{x\in X:\varphi_i(x)\leq u_i,\
i=1,\dots,n\}$.

The definition below extends the concept of the persistent
homology group to a multidimensional setting.

\begin{defn}
Let $\pi^{(\vec u,\vec v)}_k:\check{H}_k(X\langle\vec\p\preceq\vec
u\rangle)\rightarrow \check{H}_k(X\langle\vec\p\preceq\vec
v\rangle)$ be the homomorphism induced by the inclusion map
$\pi^{(\vec u,\vec v)}:X\langle\vec\p\preceq\vec
u\rangle\hookrightarrow X\langle\vec\p\preceq\vec v\rangle$ with
$\vec u\preceq\vec v$, where $\check{H}_k$ denotes the $k$th
\v{C}ech homology group. If $\vec u\prec\vec v$, the image of
$\pi^{(\vec u,\vec v)}_k$ is called the {\em multidimensional
$k$th persistent homology group of $(X,\vec\p)$ at $(\vec u, \vec
v)$}, and is denoted by $\check{H}_k^{(\vec u, \vec
v)}(X,\vec\p)$.
\end{defn}

In other words, the group $\check{H}_k^{(\vec u, \vec
v)}(X,\vec\p)$ contains all and only the homology classes of
cycles born before $\vec u$ and still alive at $\vec v$.

For details about \v{C}ech homology, the reader can refer to
\cite{EiSt}.

In what follows, we shall work with coefficients in a field
$\mathbb{K}$, so that homology groups are vector spaces, and hence
torsion-free. Therefore, they can be completely described by their
rank, leading to the following definition (cf. \cite{CaZo09}).

\begin{defn}[{\boldmath $k$}th rank invariant]\label{Rank}
Let $X$ be a triangulable space, and $\fr:X\to\R^n$ a continuous
function. Let $k\in\mathbb{Z}$. The {\em $k$th rank invariant of
the pair $(X,\fr)$ over a field $\mathbb{K}$} is the function
$\rho_{(X,\fr),k}:\Delta^+\to\N\cup\{\infty\}$ defined as
$$
\rho_{(X,\fr),k}(\vec u,\vec v)=\rank \,\pi^{(\vec u,\vec v)}_k.
$$
\end{defn}

By the rank of a homomorphism we mean the dimension of its image.
We shall prove in Lemma \ref{Finiteness} that, under our
assumptions on $X$ and $\vec\p$, the value $\infty$ is never
attained.

The main goal of this paper is to prove the following result.

\begin{thm}[Multidimensional Stability Theorem]\label{Multidimensional}
Let $X$ be a triangulable space. For every $k\in\mathbb{Z}$, there
exists a distance $D_{match}$ on the set
$\{\rho_{(X,\fr),k}\,|\,\vec\p:X\to\R^n\mbox{ continuous}\}$ such
that
\begin{eqnarray*}
D_{match}\left(\rho_{(X,\fr),k},\rho_{(X,\vec\s),k}\right)\leq\max_{x\in
X}\|\fr(x)-\vec\s(x)\|_{\infty}.
\end{eqnarray*}
\end{thm}

The construction of $D_{match}$ will be given in the course of the
proof of the theorem.

\subsection{Idea of the proof}

Since the proof of Theorem  \ref{Multidimensional} is quite
technical, for the reader's convenience, we now summarize it in
its essential ideas. We inform the reader that most of the
intermediate results needed to prove Theorem
\ref{Multidimensional}, as well as the overall flow of the proof,
perfectly match well-established results developed in Size Theory.
For these reasons, whenever this happens, details will be skipped,
providing the reader with suitable references. However, in the
present paper, terminology will stick to that of Persistence
Theory as much as possible.

The key idea is that a foliation in half-planes of $\Delta^+$ can
be given, such that the restriction of the multidimensional rank
invariant to these half-planes turns out to be a one-dimensional
rank invariant in two scalar variables. This approach implies that
the comparison of two multidimensional rank invariants can be
performed leaf by leaf by measuring the distance of appropriate
one-dimensional rank invariants. Therefore the stability of
multidimensional persistence is a consequence of the
one-dimensional persistence stability.

As is well known, in the case of tame functions, the
one-dimensional persistence stability is obtained by considering
the bottleneck distance between persistence diagrams, i.e. finite
collections of points (with multiplicity) lying in $\R^2$ above
the diagonal \cite{CoEdHa07}. We show that the same approach works
also for continuous functions, not necessarily tame. We recall
that a filtrating function is said to be a tame function if it has
a finite number of homological critical values, and the homology
groups of all the lower level sets are finitely generated. The
main technical problem in working with continuous functions,
instead of tame functions, is that persistence diagrams may have
infinitely many points, accumulating onto the diagonal. This
difficulty is overcome following the same arguments used in
\cite{dAFrLa} for proving stability of size functions.

\section{Basic properties of the rank invariant and one-dimensional
stability for continuous functions}

The main result of this section is the stability of the
one-dimensional rank invariant for continuous functions (Theorem
\ref{1-dim}). It generalizes the main theorem in \cite{CoEdHa07}
that requires continuous tame functions. Its proof relies on a
number of basic simple properties of rank invariants that are
completely analogous to those used in \cite{dAFrLa,FrLa01} to
prove the stability of size functions.

\subsection{Properties of the multidimensional
rank invariant}

The next Lemma \ref{Finiteness} ensures that the multidimensional
$k$th rank invariant (Definition \ref{Rank}) is finite even
dropping the tameness condition requested in \cite{CoEdHa07}. We
underline that the rank invariant finiteness is not obvious from
the triangulability of the space. Indeed, the lower level sets
with respect to a continuous function are not necessarily
triangulable spaces.

\begin{lem}[Finiteness]\label{Finiteness}
Let $X$ be a triangulable space, and $\fr:X\to\R^n$ a continuous
function. Then, for every $(\vec u,\vec v)\in\Delta^+$, it holds
that $\rho_{(X,\fr),k}(\vec u,\vec v)<+\infty.$
\end{lem}
\begin{proof}

Since $X$ is triangulable, we can assume that it is the support of
a simplicial complex $K$ and that a distance $d$ is defined on
$X$.

Let us fix $(\vec u,\vec v)\in\Delta^+$, and choose a real number
$\varepsilon>0$ such that, setting
$\vec\varepsilon=(\varepsilon,\dots,\varepsilon)$, $\vec
u+2\vec\varepsilon\prec\vec v$.

We now show that there exist a function $\vec\s:X\to\R^n$, a
subdivision $K''$ of $K$, and a triangulation $L''$ of
$\vec\s(X)$, such that $(i)$ the triple $(\vec\s,K'',L'')$ is
simplicial, and $(ii)$ $\max_{x\in
X}\|\fr(x)-\vec\psi(x)\|_{\infty}<\varepsilon$.

Indeed, by the uniform continuity of each component $\varphi_i$ of
$\fr$, there exists a real number $\delta>0$ such that, for
$i=1,\dots,n$, $|\varphi_i(x)-\varphi_i(x')|<\varepsilon$, for
every $x,x'\in X$ with $d(x,x')<\delta$. We take a subdivision
$K'$ of $K$ such that $\mathrm{mesh}(K')<\delta$, and define
$\vec\psi(x)=\fr(x)$ for every vertex $x$ of $K'$. Next, we
consider the linear extension of $\vec\psi$ to the other simplices
of $K'$. In this way, $\vec\psi$ is linear on each simplex of
$K'$.

Since $\vec\s$ is piecewise linear, $\vec\s(X)$ is the underlying
space of a simplicial complex $L'$. By taking suitable
subdivisions $K''$ of $K'$ and $L''$ of $L'$, $\vec\s$ also sends
simplices into simplices and therefore $(\vec\s,K'',L'')$ is
simplicial (cf. \cite[Thm. 2.14]{RoSa72}). This proves $(i)$.

To see $(ii)$, let us consider a point $x$ belonging to a simplex
in $K'$, of vertices $v_1,\dots,v_r$. Since
$x=\sum_{i=1}^r\lambda_i\cdot v_i$, with
$\lambda_1,\dots,\lambda_r\geq 0$ and $\sum_{i=1}^r\lambda_i=1$,
and $\vec\s$ is linear on each simplex, it follows that
$\big\|\vec\p(x)-\vec\s(x)\big\|_{\infty}=\big\|\vec\p(x)-\sum_{i=1}^r\lambda_i\cdot\vec\s(v_i)\big\|_{\infty}=
\big\|\vec\p(x)-\sum_{i=1}^r\lambda_i\cdot\vec\p(v_i)\big\|_{\infty}=
\big\|\sum_{i=1}^r\lambda_i\cdot\vec\p(x)-\sum_{i=1}^r\lambda_i\cdot\vec\p(v_i)\big\|_{\infty}\leq
\sum_{i=1}^r\lambda_i\big\|\vec\p(x)-\vec\p(v_i)\big\|_{\infty}<\varepsilon$.

We now prove that, since $(\vec\s,K'',L'')$ is simplicial, it
holds that $\check{H}_k(X\langle\vec\psi\preceq\vec
u+\vec\varepsilon\rangle)$ is finitely generated. Indeed, since
the intersection between a simplex and a half-space is
triangulable, there exists a subdivision $L'''$ of $L''$ such that
$\vec\s(X)\cap\{\vec x\in\R^n:\vec x\preceq\vec
u+\vec\varepsilon\}$ is triangulated by a subcomplex of $L'''$. 
By \cite[Lemma 2.16]{RoSa72}, there is a subdivision $K'''$ of
$K''$ such that $(\vec\s,K''',L''')$ is simplicial. It follows
that $X\langle\vec\s\preceq\vec u+\vec\varepsilon\rangle$ is
triangulable, and hence $\check{H}_k(X\langle\vec\s\preceq\vec
u+\vec\varepsilon\rangle)$ is finitely generated.

Since $\vec u+2\vec\varepsilon\prec\vec v$ and $\max_{x\in
X}\|\fr(x)-\vec\psi(x)\|_{\infty}<\varepsilon$, we have the
inclusions $X\langle\fr\preceq\vec u\rangle\stackrel{i}{\to}
X\langle\vec\psi\preceq \vec
u+\vec\varepsilon\rangle\stackrel{j}{\to} X\langle\fr\preceq \vec
v\rangle$, inducing the homomorphisms
$\check{H}_k(X\langle\fr\preceq\vec u\rangle)\stackrel{i_k}{\to}
\check{H}_k(X\langle\vec\psi\preceq \vec
u+\vec\varepsilon\rangle)\stackrel{j_k}{\to}
\check{H}_k(X\langle\fr\preceq \vec v\rangle)$. By recalling that
$\check{H}_k(X\langle\vec\psi\preceq \vec
u+\vec\varepsilon\rangle)$ is finitely generated, and since
$\textrm{rank}(j_k\circ i_k)\leq\textrm{rank}(j_k)$, we obtain the
claim.
\end{proof}

The following Lemma \ref{Monotonicity} and Lemma \ref{Diagonal}
generalize to the multidimensional setting analogous results valid
for $n=1$. We omit the trivial proof of Lemma \ref{Monotonicity}.

\begin{lem}[Monotonicity]\label{Monotonicity}
$\rho_{(X,\fr),k}(\vec u,\vec v)$ is non-decreasing in the
variable $\vec u$ and non-increasing in the variable $\vec v$.
\end{lem}

\begin{lem}[Diagonal Jump]\label{Diagonal}
Let $X,Y$ be two triangulable spaces, and $f:X\to Y$ a
homeomorphism. Let $\fr:X\to\R^n$, $\vec\psi:Y\to\R^n$ be
continuous functions such that $\max_{x\in
X}\|\fr(x)-\vec\psi\circ f(x)\|_{\infty}\leq h$. Then, setting
$\vec h=(h,\dots,h)$, for every $(\vec u,\vec v)\in\Delta^+$, we
have $\rho_{(X,\fr),k}(\vec u-\vec h,\vec v+\vec
h)\leq\rho_{(Y,\vec\psi),k}(\vec u ,\vec v)$.
\end{lem}
\begin{proof}
Since $\|\fr-\vec\psi\circ f\|_{\infty}\leq h$, we have the
following commutative diagram

\begin{centering}
\hfill \xymatrix {\check{H}_k(X\langle\fr\preceq\vec u-\vec
h\rangle)\ar[r]^-{i_k} \ar[d]&\check{H}_k(X\langle\fr\preceq \vec
v+\vec h\rangle)
\\\check{H}_k(Y\langle\vec\psi\preceq\vec u\rangle)\ar[r]^-{j_k}&\check{H}_k(Y\langle\vec\psi\preceq \vec
v\rangle)\ar[u],}\hfill
\end{centering}

\noindent where $i_k$ and $j_k$ are induced by inclusions, and the
vertical homomorphisms are induced by restrictions of $f$ and
$f^{-1}$, respectively. Thus the claim follows by observing that
$\textrm{rank}\ i_k\leq\textrm{rank}\ j_k$.
\end{proof}

\subsection{Properties of the one-dimensional rank invariant}

Now we confine ourselves to the case $n=1$. Therefore, for the
sake of simplicity, the symbols $\fr, \vec u,\vec v$ will be
replaced by $\p, u, v$, respectively. We remark that $\Delta^+$
reduces to be the set $\{(u,v)\in\R^2:u<v\}$. Moreover, we use the
following notations: $\Delta=\partial\Delta^+$,
$\Delta^*=\Delta^+\cup\{(u,\infty):u\in\R\}$, and
$\bar\Delta^*=\Delta^*\cup\Delta$. Finally, we write
$\|\p\|_\infty$ for $\max_{x\in X}|\p(x)|$.

\subsubsection{Right-continuity of the one-dimensional rank
invariant} \label{Sect:Right-continuity}

In what follows we shall prove that, using \v{C}ech homology, the
one-dimensional rank invariant is right-continuous with respect to
both its variables $u$ and $v$, i.e. $\lim_{u\to\bar
u^+}\rho_{(X,\p),k}(u,v)=\rho_{(X,\p),k}(\bar u,v)$ and
$\lim_{v\to\bar v^+}\rho_{(X,\p),k}(u,v)=\rho_{(X,\p),k}(u,\bar
v)$. This property will be necessary to completely characterize
the rank invariant by a persistence diagram, a descriptor whose
definition will be recalled later in this section. In the absence
of right-continuity, persistence diagrams describe rank invariants
only almost everywhere, thus justifying the use of \v{C}ech
homology in this context.

The next example shows that the right-continuity in the variable
$u$ does not always hold when persistent homology groups are
defined using simplicial or singular homology, even under the
tameness assumption.

\begin{ex}\label{WarsawCircle}
Let $X$ be a closed rectangle of $\R^2$ containing a Warsaw circle
(see Figure \ref{warsaw}). Let also $\p:X\to\R$ be the Euclidean
distance from the Warsaw circle.
\begin{figure}[h]
\begin{center}
\psfrag{X}{$X$}\psfrag{X'}{$\tilde{X}$}
\includegraphics[width=5cm]{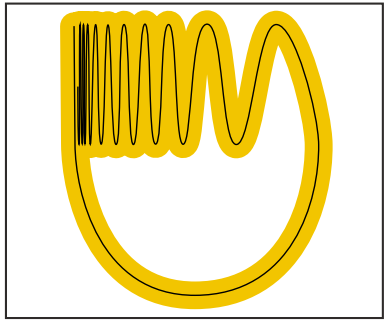}
\caption{A lower level set $X\langle\p\leq u\rangle$, for a
sufficiently small $u>0$, as considered in Example
\ref{WarsawCircle}, corresponds to a dilation (shaded) of our
Warsaw circle.}\label{warsaw}
\end{center}
\end{figure}

It is easy to see that $\p$ is tame on $X$ (with respect to both
singular and \v{C}ech homology). Moreover, the rank of the
singular persistent homology group $H_1^{(u,v)}(X,\p)$ is equal to
$1$ for $v>u>0$ and $v$ sufficiently small, whereas it is equal to
$0$ when $u=0$, showing that singular persistent homology is not
right-continuous in the variable $u$.
\end{ex}

Analogously, it is possible to show that simplicial or singular
homology do not ensure the right-continuity in the variable $v$
(see Appendix \ref{spherePatologic}).\\

Let us fix two real numbers $\bar u<\bar v$ and, for $\bar
u<u'\leq u''<\bar v$, consider the following commutative diagram

\begin{eqnarray}\label{ComDiag1}
\begin{array}{c}
\begin{centering}
\hfill \xymatrix {\check{H}_k(X\langle\p\leq
u'\rangle)\ar[rr]^-{\pi_k^{(u',u'')}} \ar[d]_{\pi_k^{(u',\bar
v)}}&&\check{H}_k(X\langle\p\leq
u''\rangle)\ar[d]_{\pi_k^{(u'',\bar v)}}
\\ \check{H}_k(X\langle\p\leq
\bar v\rangle)\ar[rr]^-{id}&&\check{H}_k(X\langle\p\leq \bar
v\rangle).}\hfill
\end{centering}
\end{array}
\end{eqnarray}

\noindent By recalling that
$\check{H}_k^{(u,v)}(X,\p)=\textrm{im}\ \pi_k^{(u,v)}$, from the
above diagram (\ref{ComDiag1}), it is easy to see that each
$\pi_k^{(u',u'')}$ induces the inclusion map
$\sigma^{(u',u'')}_k:\check{H}_k^{(u',\bar
v)}(X,\p)\to\check{H}_k^{(u'',\bar v)}(X,\p)$. The following Lemma
\ref{RightConstancyU} states that, for every $u''\geq u'>\bar u$,
with $u''$ sufficiently close to $\bar u$, the maps
$\sigma_k^{(u',u'')}$ are all isomorphisms.

\begin{lem}\label{RightConstancyU}
Let $(\bar u,\bar v)\in\Delta^+$, and let
$\sigma^{(u',u'')}_k:\check{H}_k^{(u',\bar
v)}(X,\p)\to\check{H}_k^{(u'',\bar v)}(X,\p)$ be the inclusion of
vector spaces induced by the map $\pi^{(u',u'')}_k$. Then there
exists $\hat u$, with $\bar u<\hat u<\bar v$, such that the maps
$\sigma^{(u',u'')}_k$ are isomorphisms for every $u',u''$ with
$\bar u<u'\leq u''\leq\hat u$.
\end{lem}
\begin{proof}
We observe that the maps $\sigma^{(u',u'')}_k$ are injective by
construction (see diagram (\ref{ComDiag1})). Moreover, by the
Finiteness Lemma~\ref{Finiteness} and the Monotonicity
Lemma~\ref{Monotonicity}, there exists $\hat u$, with $\bar u<\hat
u<\bar v$, such that $\rho_{(X,\p),k}(u',\bar v):=\textrm{rank
}\check{H}_k^{(u',\bar v)}(X,\p)$ is finite and equal to
$\rho_{(X,\p),k}(u'',\bar v):=\textrm{rank }\check{H}_k^{(u'',\bar
v)}(X,\p)$ whenever $\bar u<u'\leq u''\leq\hat u$. This implies
that $\sigma^{(u',u'')}_k$ are isomorphisms.
\end{proof}
Analogously, by considering the commutative diagram

\begin{centering}
\hfill \xymatrix {\check{H}_k(X\langle\p\leq \bar
u\rangle)\ar[rr]^-{id} \ar[d]_{\pi_k^{(\bar u,
v')}}&&\check{H}_k(X\langle\p\leq \bar
u\rangle)\ar[d]_{\pi_k^{(\bar u,v'')}}
\\ \check{H}_k(X\langle\p\leq
v'\rangle)\ar[rr]^-{\pi_k^{(v',v'')}}&&\check{H}_k(X\langle\p\leq
v''\rangle),}\hfill
\end{centering}

\noindent we obtain induced maps
$\tau^{(v',v'')}_k:\check{H}_k^{(\bar
u,v')}(X,\p)\to\check{H}_k^{(\bar u,v'')}(X,\p)$, and prove that
they are isomorphisms whenever $v',v''$ are sufficiently close to
$\bar v$, with $\bar v<v'\leq v''$.

\begin{lem}\label{RightConstancyV}
Let $(\bar u,\bar v)\in\Delta^+$, and let
$\tau^{(v',v'')}_k:\check{H}_k^{(\bar
u,v')}(X,\p)\to\check{H}_k^{(\bar u,v'')}(X,\p)$ be the
homomorphism of vector spaces induced by the map
$\pi^{(v',v'')}_k$. Then there exists $\hat v>\bar v$ such that
the homomorphisms $\tau^{(v',v'')}_k$ are isomorphisms for every
$v',v''$ with $\bar v<v'\leq v''\leq\hat v$.
\end{lem}
\begin{proof}
The proof is essentially the same as that of Lemma
\ref{RightConstancyU}, after observing that the maps
$\tau^{(v',v'')}_k$ are surjections between vector spaces of the
same finite dimension.
\end{proof}

\begin{lem}[Right-Continuity]\label{Right}
$\rho_{(X,\p),k}(u,v)$ is right-continuous with respect to both
the variables $u$ and $v$.
\end{lem}

\begin{proof}
In order to prove that $\lim_{u\to\bar
u^+}\rho_{(X,\p),k}(u,v)=\rho_{(X,\p),k}(\bar u,v)$, by the
Monotonicity Lemma \ref{Monotonicity}, it will suffice to show
that $\check{H}_k^{(\bar u,\bar v)}(X,\p)\cong\check{H}_k^{(\hat
u,\bar v)}(X,\p)$, where $\hat u$ is taken as in Lemma
\ref{RightConstancyU}. To this end, we consider the following
sequence of isomorphisms

{\setlength\arraycolsep{2pt}
\begin{eqnarray*}\check{H}_k^{(\bar u,\bar v)}(X,\p)&=&\textrm{im
}\pi_k^{(\bar u,\bar v)}\cong\textrm{im
}\underset{\leftarrow}\lim\,\pi_k^{(u',\bar
v)}\\
&\cong&\underset{\leftarrow}\lim\,\textrm{im }\pi_k^{(u',\bar
v)}=\underset{\leftarrow}\lim\,\check{H}_k^{(u',\bar
v)}(X,\p)\cong \check{H}_k^{(\hat u,\bar v)}(X,\p).
\end{eqnarray*}}

Let us now show how these equivalences can be obtained.

By the continuity of \v{C}ech Theory (cf. \cite[Thm. X,
3.1]{EiSt}), it holds that $\textrm{im }\pi_k^{(\bar u,\bar
v)}\cong\textrm{im }\underset{\leftarrow}\lim\,\pi_k^{(u',\bar
v)}$, where $\underset{\leftarrow}\lim\,\pi_k^{(u',\bar v)}$ is
the inverse limit of the inverse system of homomorphisms
$\pi_k^{(u',\bar v)}:\check{H}_k(X\langle\p\leq
u'\rangle)\to\check{H}_k(X\langle\p\leq \bar v\rangle)$ over the
directed set $\{u'\in\R:\bar u<u'\leq\hat u\}$ decreasingly
ordered.

Moreover, since the inverse limit of vector spaces is an exact
functor, it preserves epimorphisms and hence images. Therefore, it
holds that $\textrm{im }\underset{\leftarrow}\lim\,\pi_k^{(u',\bar
v)}\cong\underset{\leftarrow}\lim\,\textrm{im }\pi_k^{(u',\bar
v)}=\underset{\leftarrow}\lim\,\check{H}_k^{(u',\bar v)}(X,\p)$,
where  the last inverse limit is taken with respect to the inverse
system $\left(\check{H}_k^{(u',\bar
v)}(X,\p),\sigma^{(u',u'')}_k\right)$ over the directed set
$\{u'\in\R:\bar u<u'\leq\hat u\}$ decreasingly ordered, and
$\sigma_k^{(u',u'')}$ are the maps introduced in Lemma
\ref{RightConstancyU}.

Finally, $\underset{\leftarrow}\lim\,\check{H}_k^{(u',\bar
v)}(X,\p)\cong \check{H}_k^{(\hat u,\bar v)}(X,\p)$. Indeed,
$\underset{\leftarrow}\lim\,\check{H}_k^{(u',\bar v)}(X,\p)$ is
the inverse limit of a system of isomorphic vector spaces by Lemma
\ref{RightConstancyU}.

Analogously for the variable $v$, applying Lemma
\ref{RightConstancyV}.
\end{proof}

\subsubsection{Stability of the one-dimensional rank invariant}

The following Lemma \ref{Prelocalization} and Lemma \ref{Jump} can
be proved in the same way as the analogous results holding when
$k=0$ (see \cite{dAFrLa}).

\begin{lem}\label{Prelocalization}
The following statements hold:
\begin{enumerate}
    \item[$(i)$] For every $u<\min\p$, $\rho_{(X,\varphi),k}(u,v)=0$.
    \item[$(ii)$] For every $v\ge\max\p$, $\rho_{(X,\varphi),k}(u,v)$ is equal to the number of classes in $\check{H}_k(X)$,
     having at least one representative in $X\langle\p\leq
u\rangle$.
\end{enumerate}
\end{lem}

Since, for $u_1\le u_2<v_1\le v_2$, the number of homology classes
born between $u_1$ and $u_2$ and still alive at $v_1$ is certainly
not smaller than those still alive at $v_2$, we have the next
result.

\begin{lem}[Jump Monotonicity]\label{Jump}
Let $u_1,u_2,v_1,v_2$ be real numbers such that
$u_1\le u_2<v_1\le v_2$. It holds that
$$\rho_{(X,\varphi),k}(u_2,v_1)-\rho_{(X,\varphi),k}(u_1,v_1)\ge
\rho_{(X,\varphi),k}(u_2,v_2)-\rho_{(X,\varphi),k}(u_1,v_2).$$
\end{lem}

Lemma \ref{Jump} justifies the following definitions of
multiplicity. Since we are working with continuous instead of tame
functions, we adopt the definitions introduced in \cite{FrLa01}
rather than those of \cite{CoEdHa07}. Although based on the same
idea, the difference relies in the computation of multiplicity on
a varying grid, instead of a fixed one. So we can work with an
infinite number of (possibly accumulating) points of positive
multiplicity. Due to the lack of a well-established terminology
for points with positive multiplicity, we call them {\em
cornerpoints}, which is normally used for size functions.

\begin{defn}[Proper cornerpoint]\label{Proper}
For every point $p=(u,v)\in\Delta^+$, we define the number
$\mu_k(p)$ as the minimum over all the positive real numbers
$\varepsilon$, with $u+\varepsilon<v-\varepsilon$, of
$$
\rho_{(X,\p),k}(u+\varepsilon,v-\varepsilon)-\rho_{(X,\p),k}(u-\varepsilon,v-\varepsilon)
-\rho_{(X,\p),k}(u+\varepsilon,v+\varepsilon)+\rho_{(X,\p),k}(u-\varepsilon,v+\varepsilon).
$$
The number $\mu_k(p)$ will be called the \emph{multiplicity} of
$p$ for $\rho_{(X,\p),k}$. Moreover, we shall call a {\em proper
cornerpoint for $\rho_{(X,\p),k}$} any point $p\in\Delta^+$ such
that the number $\mu_k(p)$ is strictly positive.
\end{defn}

\begin{defn}[Cornerpoint at infinity]\label{Cornerpoint}
For every vertical line $r$, with equation $u=\bar u$, $\bar
u\in\R$, let us identify $r$ with $(\bar u,\infty)\in\Delta^*$,
and define the number $\mu_k(r)$ as the minimum over all the
positive real numbers $\varepsilon$, with $\bar
u+\varepsilon<1/\varepsilon$, of
$$
\rho_{(X,\p),k}\left(\bar
u+\varepsilon,\frac{1}{\varepsilon}\right)-\rho_{(X,\p),k}\left(\bar
u-\varepsilon,\frac{1}{\varepsilon}\right).
$$
The number $\mu_k(r)$ will be called the \emph{multiplicity} of
$r$ for $\rho_{(X,\p),k}$. When this finite number is strictly
positive, we call $r$ a {\em cornerpoint at infinity for
$\rho_{(X,\p),k}$}.
\end{defn}

The concept of cornerpoint allows us to introduce a representation
of the rank invariant, based on the following definition
\cite{CoEdHa07}.

\begin{defn}[Persistence diagram]
The {\em persistence diagram} $D_k(X,\p)\subset\bar\Delta^*$ is
the multiset of all cornerpoints (both proper and at infinity) for
$\rho_{(X,\p),k}$, counted with their multiplicity, union the
points of $\Delta$, counted with infinite multiplicity.
\end{defn}

In order to show that persistence diagrams completely describe
rank invariants, we give some technical results concerning
cornerpoints.

The Monotonicity Lemma \ref{Monotonicity}, and Lemmas \ref{Right}
(Right-Continuity) and \ref{Jump} (Jump Monotonicity) imply the
following result, by the same arguments as in \cite{FrLa01}.

\begin{prop}[Propagation of Discontinuities]\label{Propagation}
If $\bar p=(\bar u,\bar v)$ is a proper cornerpoint for
$\rho_{(X,\p),k}$, then the following statements hold:
\begin{enumerate}
\item[$(i)$] If $\bar u \leq u < \bar v$, then $\bar v$ is a
discontinuity point for $\rho_{(X,\p),k}(u,\cdot)$; \item[$(ii)$]
If $\bar u < v < \bar v$, then $\bar u$ is a discontinuity point
for $\rho_{(X,\p),k}(\cdot,v)$.
\end{enumerate}
If $\bar r = (\bar u,\infty)$ is a cornerpoint at infinity for
$\rho_{(X,\p),k}$, then it holds that
\begin{enumerate}
\item[$(iii)$] If $\bar u<v$, then $\bar u$ is a discontinuity
point for $\rho_{(X,\p),k}(\cdot,v)$.
\end{enumerate}
\end{prop}

We observe that any open arcwise connected neighborhood in
$\Delta^+$ of a discontinuity point for $\rho_{(X,\p),k}$ contains
at least one discontinuity point in the variable $u$ or $v$.
Moreover, as a consequence of the Jump Monotonicity Lemma
\ref{Jump}, discontinuity points in the variable $u$ propagate
downwards, while discontinuity points in the variable $v$
propagate rightwards. So, by applying the Finiteness Lemma
\ref{Finiteness}, we obtain the proposition below.

\begin{prop}\label{Weps}
Let $k\in\mathbb{Z}$. For every point $\bar p=(\bar u,\bar
v)\in\Delta^+$, a real number $\varepsilon>0$ exists such that the
open set $$W_{\varepsilon}(\bar p)=\{(u,v)\in\R^2:|u-\bar
u|<\varepsilon,|v-\bar v|<\varepsilon, u\neq\bar u, v\neq\bar
v\}$$ is contained in $\Delta^+$, and does not contain any
discontinuity point for $\rho_{(X,\p),k}$.
\end{prop}

As a simple consequence of Lemma \ref{Prelocalization} and
Proposition \ref{Propagation} (Propagation of Discontinuities), we
have the following proposition.

\begin{prop}[Localization of Cornerpoints]\label{Localization}
If $\bar p=(\bar u,\bar v)$ is a proper cornerpoint for
$\rho_{(X,\varphi),k}$, then $\bar
p\in\{(u,v)\in\Delta^+:\min\p\leq u<v\leq\max\p\}$.
\end{prop}

By applying Proposition \ref{Propagation} and Proposition
\ref{Weps}, and recalling the finiteness of $\rho_{(X,\varphi),k}$
(Finiteness Lemma \ref{Finiteness}), it is easy to prove the
following result.

\begin{prop}[Local Finiteness of Cornerpoints]\label{Local}
For each strictly positive real number $\varepsilon$,
$\rho_{(X,\varphi),k}$ has, at most, a finite number of
cornerpoints in $\{(u,v)\in\R^2:u+\varepsilon< v\}$.
\end{prop}

We observe that it is easy to provide examples of persistence
diagrams containing an infinite number of proper cornerpoints,
accumulating onto the diagonal $\Delta$.

The following Theorem \ref{k-triangle} shows that persistence
diagrams uniquely determine one-dimensional rank invariants (the
inverse also holds by definition of persistence diagram). It is a
consequence of the definitions of multiplicity (Definitions
\ref{Proper} and \ref{Cornerpoint}), together with the previous
results about cornerpoints, and the Right-Continuity Lemma
\ref{Right}, in the same way as done in \cite{FrLa01} for size
functions. We remark that a similar result was given in
\cite{CoEdHa07}, under the name of \emph{$k$-triangle Lemma}. Our
Representation Theorem differs from the $k$-triangle Lemma in two
respects. Firstly, our hypotheses on the function $\p$ are weaker.
Secondly, the $k$-triangle Lemma focuses not on all the set
$\Delta^+$, but only on the points having coordinates that are not
homological critical values.

\begin{thm}[Representation Theorem]\label{k-triangle}
For every $(\bar u,\bar v)\in\Delta^+$, we have
$$
\rho_{(X,\p),k}(\bar u,\bar v)=\sum_{(u,v)\in\Delta^*\atop
u\leq\bar u,\,v>\bar v}\mu_k((u,v)).
$$
\end{thm}

As a consequence of the Representation Theorem \ref{k-triangle}
any distance between persistence diagrams induces a distance
between one-dimensional rank invariants. This justifies the
following definition \cite{CoEdHa07,dAFrLa}.

\begin{defn}[Matching distance]
Let $X$ be a triangulable space endowed with continuous functions
$\p,\s:X\to\R$. The {\em matching distance} $d_{match}$ between
$\rho_{(X,\p),k}$ and $\rho_{(X,\s),k}$ is equal to the bottleneck
distance between $D_k(X,\p)$ and $D_k(Y,\s)$, i.e.
\begin{eqnarray}\label{DistMatch}
d_{match}\left(\rho_{(X,\p),k},\rho_{(X,\s),k}\right)=\inf_{\gamma}\max_{p\in
D_k(X,\p)}\|p-\gamma(p)\|_{\widetilde{\infty}},
\end{eqnarray}
where $\gamma$ ranges over all multi-bijections between
$D_k(X,\p)$ and $D_k(X,\s)$, and for every $p=(u,v),q=(u',v')$ in
$\Delta^*$,
$$
\|p-q\|_{\widetilde{\infty}}=
\min\left\{\max\left\{|u-u'|,|v-v'|\right\},\max\left\{\frac{v-u}{2},\frac{v'-u'}{2}\right\}\right\},
$$
with the convention about points at infinity that
$\infty-y=y-\infty=\infty$ when $y\neq\infty$, $\infty-\infty=0$,
$\frac{\infty}{2}=\infty$, $|\infty|=\infty$, $\min\{c,\infty\}=c$
and $\max\{c,\infty\}=\infty$.
\end{defn}

In plain words, $\|\cdot\|_{\widetilde{\infty}}$ measures the
pseudodistance between two points $p$ and $q$ as the minimum
between the cost of moving one point onto the other and the cost
of moving both points onto the diagonal, with respect to the
max-norm and under the assumption that any two points of the
diagonal have vanishing pseudodistance. We observe that in
(\ref{DistMatch}) we can write $\max$ instead of $\sup$, because
by Proposition \ref{Localization} (Localization of Cornerpoints)
proper cornerpoints belong to a compact subset of the closure of
$\Delta^+$.

We are now ready to give the one-dimensional stability theorem for
the rank invariant with continuous functions. The proof relies on
a cone construction. The rationale behind this construction is to
directly apply the arguments used in \cite{dAFrLa} for size
functions, eliminating cornerpoints at infinity, whose presence
would require us to modify all the proofs.

This stability theorem is a different result from the one given in
\cite{CoEdHa07}, weakening the tameness requirement to continuity,
and actually solving one of the open problems posed in that work
by the authors.

\begin{thm}[One-Dimensional Stability Theorem]\label{1-dim}
Let $X$ be a triangulable space, and $\p,\s:X\to\R$ two continuous
functions. Then
$d_{match}(\rho_{(X,\p),k},\rho_{(X,\s),k})\leq\|\p-\s\|_\infty$.
\end{thm}
\begin{proof}
In what follows we can assume that $X$ is connected. Indeed, if
$X$ has $r$ connected components $C_1,\dots,C_r$, then the claim
can be proved by induction after observing that
$D_k(X,\p)=\bigcup_{i=1}^r D_k(C_i,\p_{|C_i})$, for every
$k\in\mathbb{Z}$.

For $k=0$, the claim has been proved in \cite[Thm. 25]{dAFrLa}.

Let us now assume $k>0$. Consider the cone on $X$, $\tilde
X=(X\times I)\big/(X\times\{1\})$ (see Figure \ref{cono}).

\begin{figure}[h]
\begin{center}
\psfrag{X}{$X$}\psfrag{X'}{$\tilde{X}$}\psfrag{t=0}{$t=0$}\psfrag{t=1}{$t=1$}
\includegraphics[width=10cm]{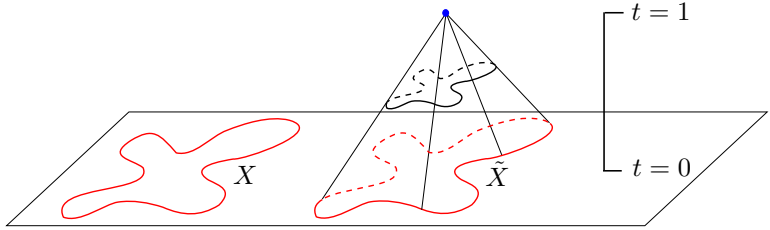}
\caption{The cone construction used in the proof of Theorem
\ref{1-dim}. The cycles in the cone are trivial.}\label{cono}
\end{center}
\end{figure}

Since $X$ is triangulable, so is $\tilde X$. We also consider the
continuous function $\tilde\p:\tilde X\to\R$ taking the class of
$(x,t)$ to the value $\p(x)\cdot(1-t)+M\cdot t$, where
$M=3\cdot(\max|\p|+\max|\s|)+1$. This choice of $M$, besides
guaranteeing that $(u,M)\in\Delta^+$ when
$u\leq\max|\p|,\max|\s|$, will be useful later.

By construction, it holds that
$$
\rho_{(\tilde X,\tilde\p),k}(u,v)=
\left\{%
\begin{array}{ll}
    \rho_{(X,\p),k}(u,v), & \hbox{if $v<M$;} \\
    0, & \hbox{if $v\geq M$.} \\
\end{array}%
\right.
$$
Indeed, it is well known that $\tilde X$ is contractible (see
\cite[Lemma 21.13]{GrHa81}), explaining why $\rho_{(\tilde
X,\tilde\p),k}(u,v)=0$ when $v\geq M$. The other part of the claim
follows from the observation that, for every $v<M$, identifying
$X\langle\p\leq v\rangle\times\{0\}$ with $X\langle\p\leq
v\rangle$, the lower level set $X\langle\p\leq v\rangle$ is a
strong deformation retract of $\tilde X\langle\tilde\p\leq
v\rangle$. To see this, it is sufficient to consider the obvious
retraction $r: (x,t)\mapsto x$ and the deformation retraction
$H:\tilde X\langle\tilde\p\leq v\rangle\times I\to \tilde
X\langle\tilde\p\leq v\rangle$, $H((x,t),s)=(x,t\cdot(1-s))$. This
yields the following commutative diagram

\begin{centering}
\hfill \xymatrix {\check{H}_k(\tilde X\langle\tilde\p\leq
u\rangle)\ar[r]^-{r_k}
\ar[d]_{\pi_k^{(u,v)}}&\check{H}_k(X\langle\p\leq
u\rangle)\ar[d]_{\pi_k^{(u,v)}}
\\ \check{H}_k(\tilde X\langle\tilde\p\leq
v\rangle)\ar[r]^-{r_k}&\check{H}_k(X\langle\p\leq
v\rangle),}\hfill
\end{centering}

\noindent where the horizontal maps are isomorphisms, so that
$\check{H}_k^{(u,v)}(\tilde
X,\tilde\p)\cong\check{H}_k^{(u,v)}(X,\p)$ when $v<M$.

Clearly, a point of $\Delta^+$ is a proper cornerpoint for
$\rho_{(X,\p),k}$ if and only if it is a proper cornerpoint for
$\rho_{(\tilde X,\tilde\p),k}$, with the ordinate strictly less
than $M$. Moreover, a point $(u,\infty)$ of $\Delta^*$ is a
cornerpoint at infinity for $\rho_{(X,\p),k}$ if and only if the
point $(u,M)\in\Delta^+$ is a proper cornerpoint for
$\rho_{(\tilde X,\tilde\p),k}$. We remark that there are no
cornerpoints $(u,v)$ for $\rho_{(\tilde X,\tilde\p),k}$ when
$\max|\p|<v<M$. Analogously, we can construct $\tilde\s:\tilde
X\to\R$ out of $\s$ with the same properties.

Following the same technical steps as in \cite[Thm. 25]{dAFrLa},
simply substituting $\rho_{(X,\p),k}$ for $\rho_{(X,\p),0}$, it is
possible to prove the inequality $d_{match}(\rho_{(\tilde
X,\tilde\p),k},\rho_{(\tilde
X,\tilde\s),k})\leq\|\tilde\p-\tilde\s\|_{\infty}$. To this end,
we need to apply Lemma \ref{Diagonal}, and Propositions
\ref{Propagation}, \ref{Weps}, \ref{Localization} and \ref{Local}.
Therefore, since $\|\tilde\p-\tilde\s\|_\infty=\|\p-\s\|_\infty$,
it is sufficient to show that
$d_{match}(\rho_{(X,\p),k},\rho_{(X,\s),k})\leq
d_{match}(\rho_{(\tilde X,\tilde\p),k},\rho_{(\tilde
X,\tilde\s),k})$.

We can prove that a multi-bijection $\tilde\gamma$ between
$D_k(\tilde X,\tilde\p)$ and $D_k(\tilde X,\tilde\s)$ exists, with
$d_{match}(\rho_{(\tilde X,\tilde\p),k},\rho_{(\tilde
X,\tilde\s),k})=\max_{\tilde p\in D_k(\tilde X,\tilde\p)}\|\tilde
p-\tilde\gamma(\tilde p)\|_{\widetilde{\infty}}$. This can be done
by applying Proposition \ref{Local}, as in \cite[Thm. 28]{dAFrLa}.
Such a $\tilde\gamma$ will be called optimal.

Since $\tilde\gamma$ is optimal, then $\tilde\gamma$ takes each
point $(u,v)\in D_k(\tilde X,\tilde\p)$, with $v=M$, to a point
$(u',v')\in D_k(\tilde X,\tilde\s)$, with $v'=M$. Indeed, if it
were not true, i.e. $\tilde\gamma((u,M))=(u',v')$ with $v'<M$,
then $v'\leq\max|\s|$, and by the choice of $M$, we would have
$\|(u,M)-\tilde\gamma((u,M))\|_{\widetilde{\infty}}\geq\|\p-\s\|_{\infty}$,
contradicting the fact that $d_{match}(\rho_{(\tilde
X,\tilde\p),k},\rho_{(\tilde
X,\tilde\s),k})\leq\|\p-\s\|_{\infty}$. Analogously for
$\tilde\gamma^{-1}$, proving that $\tilde\gamma$ maps cornerpoints
whose ordinate is smaller than $M$ into cornerpoints whose
ordinate is still below $M$.

Let us now consider an optimal multi-bijection $\tilde\gamma$
between $D_k(\tilde X,\tilde\p)$ and $D_k(\tilde X,\tilde\s)$. We
now show that there exists a multi-bijection $\gamma$ between
$D_k(X,\p)$ and $D_k(X,\s)$, such that $\max_{p\in
D_k(X,\p)}\|p-\gamma(p)\|_{\widetilde{\infty}}=\max_{\tilde p\in
D_k(\tilde X,\tilde\p)}\|\tilde p-\tilde\gamma(\tilde
p)\|_{\widetilde{\infty}}$, thus proving that
$d_{match}(\rho_{(X,\p),k},\rho_{(X,\s),k})\leq
d_{match}(\rho_{(\tilde X,\tilde\p),k},\rho_{(\tilde
X,\tilde\s),k})$. Indeed, we can define $\gamma:D_k(X,\p)\to
D_k(X,\s)$ by setting $\gamma((u,v))=\tilde\gamma((u,v))$ if
$v<\infty$, and $\gamma((u,v))=(u',v)$, where $u'$ is the abscissa
of the point $\tilde\gamma((u,M))$, if $v=\infty$. This concludes
the proof.
\end{proof}

\section{Proof of the Multidimensional Stability Theorem \ref{Multidimensional}}

We now provide the proof of the Multidimensional Stability Theorem
\ref{Multidimensional}. It will be deduced following the same
arguments given in \cite{BiCe*08} to prove the stability of
multidimensional size functions. Proofs that are still valid for
$k>0$ without any change will be omitted.

We start by recalling that the following  parameterized family of
half-planes in $\R^n\times\R^n$ is a foliation of $\Delta ^+$.

\begin{defn}[Admissible pairs]\label{Admissible}
\label{np} For every unit vector $\vec{l}=(l_1,\ldots,l_n)$ of
$\mathbb{R}^n$ such that $l_i>0$ for $i=1,\dots,n$, and for every
vector $\vec{b}=(b_1,\ldots,b_n)$ of $\mathbb{R}^n$ such that
$\sum_{i=1}^n b_i=0$, we shall say that the pair
$(\vec{l},\vec{b})$ is \emph{admissible}. We shall denote the set
of all admissible pairs in $\R^n\times\R^n$ by $Adm_n$. Given an
admissible pair $(\vec{l},\vec{b})$, we define the half-plane
$\pi_{(\vec{l},\vec{b})}$ of $\R^n\times\R^n$ by the following
parametric equations:
$$
\left\{%
\begin{array}{ll}
    \vec u=s\vec l + \vec b\\
    \vec v=t\vec l + \vec b\\
\end{array}%
\right.
$$
for $s,t\in \R$, with $s<t$.
\end{defn}

The key property of this foliation is  that the restriction of
$\rho_{(X,\fr),k}$ to each leaf can be seen as a particular
one-dimensional rank invariant, as the following theorem states.

\begin{thm}[Reduction Theorem]\label{Reduction}
Let $(\vec{l},\vec{b})$ be an admissible pair, and $F_{(\vec
l,\vec b)}^{\fr}:X\rightarrow\R$ be defined by setting
$$
F_{(\vec l,\vec
b)}^{\fr}(x)=\max_{i=1,\dots,n}\left\{\frac{\varphi_i(x)-b_i}{l_i}\right\}\
.
$$
Then, for every $(\vec u,\vec v)=(s\vec l+\vec b,t\vec l + \vec
b)\in\pi_{(\vec{l},\vec{b})}$ the following equality holds:
$$
\rho_{(X,\fr),k}(\vec u,\vec v)=\rho_{(X,F_{(\vec l,\vec
b)}^{\fr}),k}(s,t)\ .
$$
\end{thm}

As a consequence of the Reduction Theorem \ref{Reduction}, we
observe that the identity
$\rho_{(X,\fr),k}\equiv\rho_{(X,\vec{\psi}),k}$ holds if and only
if $d_{match}(\rho_{(X,F_{(\vec l,\vec
b)}^{\fr}),k},\rho_{(X,F_{(\vec l,\vec b)}^{\vec\psi}),k})=0$, for
every admissible pair $(\vec{l},\vec{b})$.

The next theorem gives a stability result on each leaf of the
foliation. It represents an intermediate step toward the proof of
the stability of the multidimensional rank invariant across the
whole foliation.

\begin{thm}[Stability w.r.t. Function Perturbations]\label{Stability}
If $X$ is triangulable and $\fr,\vec{\psi}:X\to\R^n$ are
continuous functions, then for each admissible pair
$(\vec{l},\vec{b})$, it holds that
$$
d_{match}(\rho_{(X,F_{(\vec l,\vec b)}^{\fr}),k},\rho_{(X,F_{(\vec
l,\vec b)}^{\vec\psi}),k})\leq\frac{\max_{x\in X}
\|\fr(x)-\vec\psi(x)\|_{\infty}}{\min_{i=1,\dots,n}l_i}.
$$
\end{thm}

We are now ready to deduce the Multidimensional Stability Theorem
\ref{Multidimensional}.

\begin{proof}{\em (of Theorem \ref{Multidimensional})}
We set
$D_{match}(\rho_{(X,\fr),k},\rho_{(X,\vec\psi),k})=\sup_{(\vec
l,\vec b)\in Adm_n}\min_i l_i\cdot d_{match}(\rho_{(X,F_{(\vec
l,\vec b)}^{\fr}),k},\rho_{(X,F_{(\vec l,\vec
b)}^{\vec\psi}),k})$. Then, $D_{match}$  is a distance on
$\{\rho_{(X,\fr),k}\,|\,\vec\p:X\to\R^n\mbox{ continuous}\}$ that
clearly proves the claim.
\end{proof}

Roughly speaking, we have thus proved that small changes in the
vector-valued filtrating function induce small changes in the
associated multidimensional rank invariant, with respect to the
distance $D_{match}$.

\subsection*{Acknowledgements}
The authors  thank Francesca Cagliari (University of Bo\-lo\-gna)
and Marco Grandis (University of Genoa) for their helpful advice.
However, the authors are solely responsible for any errors.

\renewcommand{\thesection}{A}
\setcounter{equation}{0}  

\section{Appendix}\label{spherePatologic}
The next example shows that the rank invariant is not
right-continuous in the variable $v$ when singular or simplicial
homology are considered instead of \v{C}ech homology. We recall
that an example concerning the right-continuity in the variable
$u$ has been given in Example \ref{WarsawCircle}.
\begin{ex}\label{exsphere}
Let $S \subset \R^3$ be a sphere parameterized by polar
coordinates $(\alpha,\beta)$, $-\frac{\pi}{2}\leq\alpha \leq
\frac{\pi}{2}$ and $\beta \in [0, 2\pi)$. For every $\beta \in [0,
2\pi)$, consider on $S$ the paths
$\gamma_{\beta}^1:(-\frac{\pi}{2}, 0) \to S$ and
$\gamma_{\beta}^2:(0, \frac{\pi}{2}) \to S$ defined by setting,
for $i=1,2$, $\gamma^i_{\beta}(\alpha)=(\alpha', \beta')$ with
$\alpha'=\alpha$ and $\beta'=(\beta + \cot\alpha)\!\!\mod 2\pi$.
We observe that each point of the set $S^*=\{(\alpha,\beta) \in S:
\alpha \neq 0 \wedge |\alpha| \neq \frac{\pi}{2}\}$ belongs to the
image of one and only one path $\gamma^i_{\beta}$. Such curves
approach more and more a pole of the sphere on one side and the
equator, winding an infinite number of times, on the other side
(see, for instance, in Figure \ref{sphere} $(a)$, the paths
$\gamma^2_{\frac{\pi}{2}}$ and $\gamma^2_{\frac{3\pi}{2}}$ lying
in the northern hemisphere).
\begin{figure}[htbp]
\begin{center}
\psfrag{0}{$0$}\psfrag{1}{$1$}\psfrag{2}{$2$}\psfrag{u}{$u$}\psfrag{v}{$v$}
\psfrag{r}{$\rho_{(S,\p),0}$}\psfrag{D+}{$\Delta^+$}\psfrag{-e}{$\p(\bar
P)=\p(\bar Q)$}\psfrag{P}{$\bar P$}\psfrag{N}{$N$}
\begin{tabular}{cc}
\includegraphics[width=5cm]{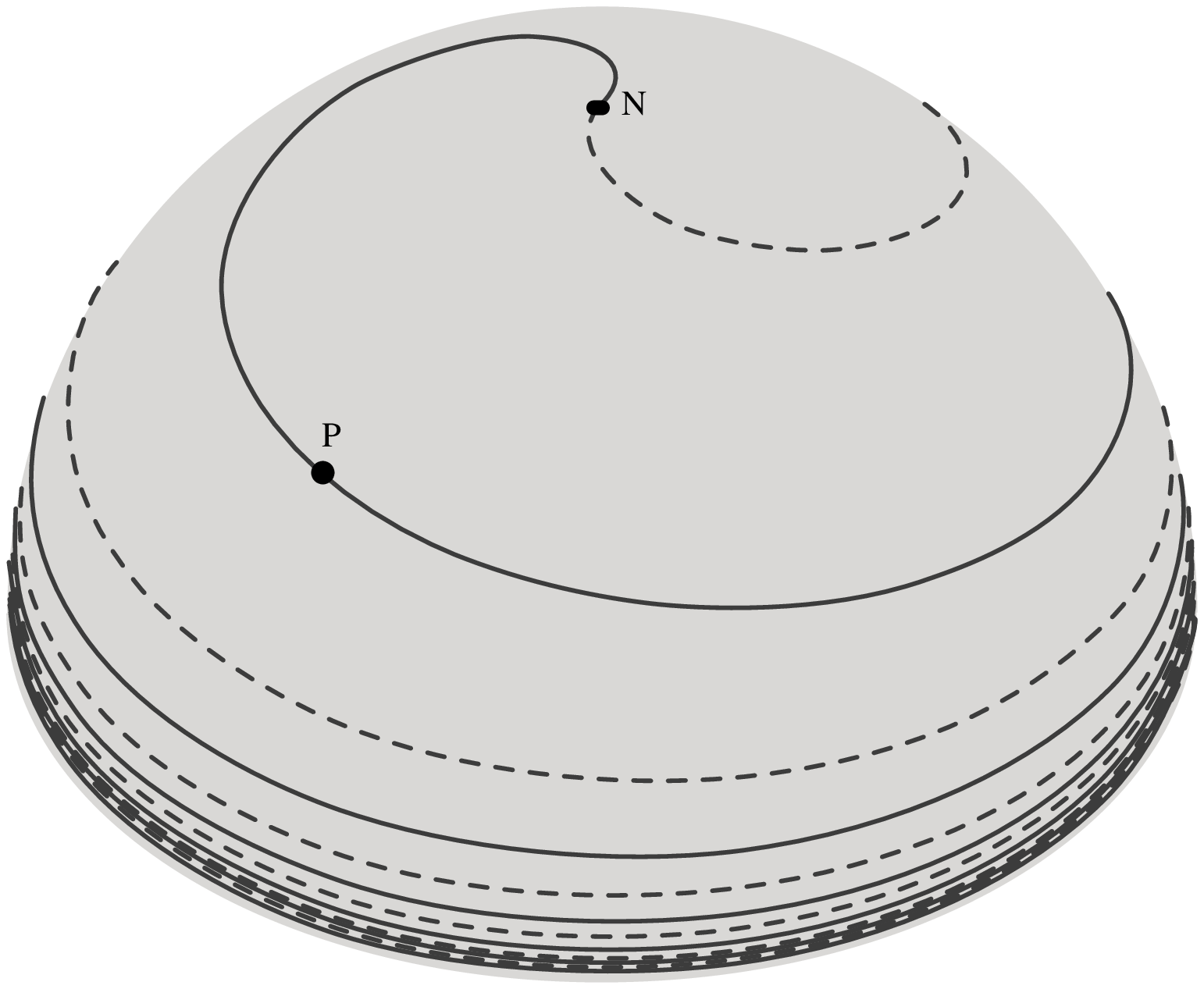}&\includegraphics[width=5cm]{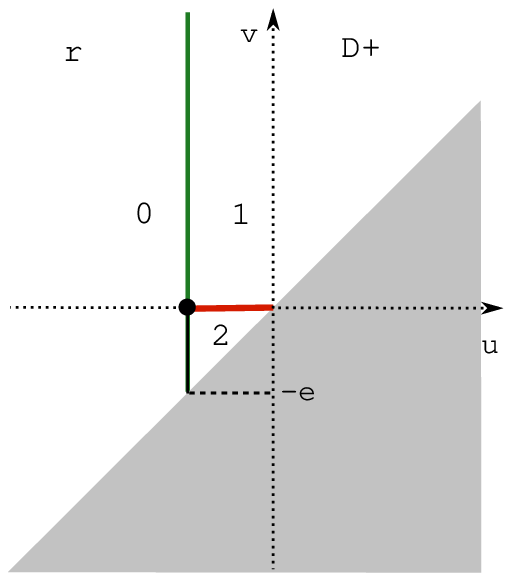}\\
$(a)$&$(b)$\\
\end{tabular}
\caption{$(a)$ Two of the paths covering the northern hemisphere
considered in Example \ref{exsphere}. $(b)$ The $0$th rank
invariant of the pair $(S,\p)$. On the discontinuity points
highlighted in bold red, the 0th rank invariant computed using
singular homology takes value equal to 2, while using \v{C}ech
homology, the value is equal to 1, showing the right-continuity in
the variable $v$.}\label{sphere}
\end{center}
\end{figure}

Then define the $C^{\infty}$ function $\p^*: S^* \to \R$ that
takes each point $P = \gamma^i_{\beta}(\alpha) \in S^*$ to the
value
$\exp\left(-\frac{1}{\alpha^2\left(\frac{\pi}{2}-|\alpha|\right)^2}\right)\sin(\beta)$.
Now extend $\p^*$ to a $C^{\infty}$ function $\p:S\to\R$ in the
unique possible way. In plain words, this function draws a ridge
for $\beta \in (0,\pi)$, and a valley for $\beta \in (\pi,2\pi)$.
Moreover, observe that the points $\bar P\equiv(\frac{\pi}{4},
\frac{3\pi}{2})$ and $\bar Q\equiv(-\frac{\pi}{4},
\frac{3\pi}{2})$ of the sphere are the unique local minimum points
of $\p$.

Let us now consider the $0$th rank invariant of the pair $(S,\p)$.
Its graph is depicted in Figure \ref{sphere} $(b)$. The points
$\bar P$ and $\bar Q$ belong to the same arcwise connected
component of the lower level set
$S\langle\p\leq\varepsilon\rangle$ for every $\varepsilon>0$,
whereas they do not for $\varepsilon=0$, since the paths
$\gamma^i_{\frac{\pi}{2}}$ ($i=1,2$) are an ``obstruction'' to
construct a continuous path from $\bar P$ to $\bar Q$. Hence, the
singular rank invariant $\rho_{(S,\p),0}$ is not right-continuous
in the second variable at $v=0$, for any $u$ with $\min\varphi<u
<0$.
\end{ex}

\bibliographystyle{abbrv}
\bibliography{multistab_arxiv}

\end{document}